\documentclass[11pt,reqno]{amsart}

\makeatletter
\usepackage{amssymb}
\usepackage{amssymb,amsmath,color,comment}
\usepackage{amsbsy}
\usepackage{amsfonts}

\def\marginpar#1{\ignorespaces}

\newtheorem{theorem}[equation]{Theorem}
\newtheorem{proposition}[equation]{Proposition}
\newtheorem{lemma}[equation]{Lemma}

\newtheorem{definition}[equation]{Definition}

\theoremstyle{definition}
\newtheorem{remark}[equation]{Remark}

\numberwithin{equation}{section}

\def\AArm{\fam0 \rm}%
\newdimen\AAdi%
\newbox\AAbo%
\def\AAk#1#2{\setbox\AAbo=\hbox{#2}\AAdi=\wd\AAbo\kern#1\AAdi{}}%

\newcommand{\BBone}{{\ensuremath{{\AArm 1\AAk{-.8}{I}I}}}}

\def\eqref#1{(\ref{#1})}
\def\eqlabel#1{\def\@currentlabel{#1}}

\def\formula#1{\def\@tempa{#1}\let\@tempb\theequation\def\theequation{%
\hbox{#1}}\def\@currentlabel{(\theequation)}$$}
\def\endformula{\leqno\hbox{(\@tempa)}$$\@ignoretrue\let\theequation\@tempb}

\def\given{\hskip5\p@\relax\vrule\@width.4\p@\hskip5\p@\relax}

\newcommand{\open}[1]{%
\par\normalfont\topsep6\p@\@plus6\p@\trivlist\item[\hskip\labelsep\itshape#1%
\@addpunct{.}]\ignorespaces}

\DeclareRobustCommand{\close}[1]{%
  \ifmmode 
  \else \leavevmode\unskip\penalty9999 \hbox{}\nobreak\hfill
  \fi
  \quad\hbox{$#1$}}

\newlength{\toskip}

\def\<{\langle}
\def\>{\rangle}

\def \R {{\mathbb R}}

\def \E {{\mathbb E}}

\def \L {{\mathbb L}}

\makeatother

\begin{document}
\date{\today}

\title[Stochastic Keller-Segel]{The 2-D stochastic Keller-Segel particle model : existence and uniqueness.}

 \author[P. Cattiaux]{\textbf{\quad {Patrick} Cattiaux $^{\spadesuit}$ \, \, }}
\address{{\bf {Patrick} CATTIAUX},\\ Institut de Math\'ematiques de Toulouse. Universit\'e de Toulouse. CNRS UMR 5219. \\ 118 route de Narbonne, F-31062 Toulouse cedex 09.}
\email{patrick.cattiaux@math.univ-toulouse.fr}

\author[L. P\'ed\`eches]{\textbf{\quad {Laure} P\'ed\`eches $^{\spadesuit}$}}
\address{{\bf {Laure} PEDECHES},\\ Institut de Math\'ematiques de Toulouse. Universit\'e de Toulouse. CNRS UMR 5219. \\ 118 route de Narbonne, F-31062 Toulouse cedex 09.}
\email{laure.pedeches@math.univ-toulouse.fr}

\maketitle
 \begin{center}

 \textsc{$^{\spadesuit}$  Universit\'e de Toulouse}
\smallskip

 \end{center}

\begin{abstract}
We introduce a stochastic system of interacting particles which is expected to furnish as the number of particles goes to infinity a stochastic approach of the $2$-D Keller-Segel model. In this note, we prove existence and some uniqueness for the stochastic model for the parabolic-elliptic Keller-Segel equation, for all regimes under the critical mass. Prior results for existence and weak uniqueness have been very recently obtained by N. Fournier and B. Jourdain \cite{FJ}.
\end{abstract}
\bigskip

\bigskip

\bigskip

\section{\bf Introduction and main results.}\label{Intro}

The (Patlak) Keller-Segel system is a tentative model to describe chemo-taxis phenomenon, an attractive chemical phenomenon between organisms. In two dimensions, the classical 2-D parabolic-elliptic Keller-Segel model reduces to a single non linear P.D.E., 
\begin{equation}\label{eqKS}
\partial_t \rho_t(x) = \Delta_x \, \rho_t(x) + \chi \, \nabla_x.((K*\rho_t)\rho_t)(x) 
\end{equation}
with some initial $\rho_0$.\\ Here $\rho:\mathbb R_+ \times \mathbb R^2 \, \to \, \mathbb R$, $\chi>0$ and $K:x \in \mathbb R^2 \, \mapsto \, \frac{x}{\parallel x\parallel^2} \, \in \mathbb R^2$ is the gradient of the harmonic kernel, i.e. $K(x) = \nabla \, \log(\parallel x\parallel)$.\\ It is not difficult to see that \eqref{eqKS} preserves positivity and mass, so that we may assume that $\rho_0$ is a density of probability, i.e. $\rho_0\geq 0$ and $\int \rho_t dx = \int \rho_0 dx =1$. For an easy comparison with the P.D.E. literature, just remember that our parameter $\chi$ satisfies $\chi = \frac{m}{2\pi}$ when $m$ denotes the total mass of $\rho$ (the parameter $2\pi$ being usually included in the definition of $K$).\\ As usual, $\rho$ is modeling a density of cells, and $c_t=K*\rho_t$ is (up to some constant) the concentration of chemo-attractant. 
\medskip

A very interesting property of such an equation is a blow-up phenomenon
\begin{theorem}\label{blow-up}
Assume that $\rho_0 \, \log \rho_0 \in \mathbb L^1(\mathbb R^2)$ and that $(1+\parallel x\parallel^2) \rho_0  \in \mathbb L^1(\mathbb R^2)$. Then if $\chi >4$, the maximal time interval of existence of a classical solution of \eqref{eqKS} is $[0,T^*)$ with $$T^* \, \leq \, \frac{1}{2\pi \, \chi \, (\chi-4)} \, \int \parallel x\parallel^2 \, \rho_0(x) \, dx \, .$$ If $\chi \leq 4$ then $T^*=+\infty$.
\end{theorem}

For this result, a wonderful presentation of what Keller-Segel models are and an almost up to date state of the art, we refer to the unpublished HDR document of Adrien Blanchet (available on Adrien's webpage   \cite{Adrien}). We also apologize for not furnishing a more complete list of references on the topic, where beautiful results were obtained by brilliant mathematicians. But the present paper is intended to be a short note.
\medskip

Actually \eqref{eqKS} is nothing else but a Mc Kean-Vlasov type equation (non linear Fokker-Planck equation if one prefers), involving a potential which is singular at $0$. Hence one can expect that the movement of a  typical cell will be given by a non-linear diffusion process
\begin{eqnarray}\label{eqnldiff}
dX_t &=& \sqrt 2 \, dB_t \, - \, \chi \, (K*\rho_t)(X_t) \, dt \, ,\\
\rho_t(x) \, dx &=& \mathcal L(X_t) \, , \nonumber
\end{eqnarray}
where $\mathcal L(X_t)$ denotes the distribution of probability of $X_t$. The natural linearization of \eqref{eqnldiff} is through the limit of a linear system of stochastic differential equations in mean field interactions given for $i=1,...,N$ by
\begin{equation}\label{eqsys}
dX_t^{i,N} = \sqrt 2 \, dB_t^{i,N} \, - \, \frac {\chi}{N} \, \sum_{j\neq i}^N \, \frac{X_t^{i,N}-X_t^{j,N}}{\parallel X_t^{i,N}-X_t^{j,N} \parallel^2} \, \, dt \, ,
\end{equation}
for a well chosen initial distribution of the $X_0^{.,N}$. Here the $B_.^{i,N}$ are for each $N$ independent standard 2-D Brownian motions. Under some exchangeability assumptions, it is expected that the distribution of any particle (say $X^{1,N}$) converges to a solution of \eqref{eqnldiff} as $N \to \infty$, yielding a solution to \eqref{eqKS}. This strategy (including the celebrated propagation of chaos phenomenon) has been well known for a long time. One can see \cite{Mel} for bounded and Lipschitz potentials, \cite{malrieu01,CGM} for unbounded potentials connected with the granular media equation.
\medskip

The goal of the present note is the study of existence, uniqueness and non explosion for the system \eqref{eqsys}. That is, this is the very first step of the whole program we have described previously. Moreover we will see how the $N$-particle system is feeling the blow-up property of the Keller-Segel equation. 
\medskip

For such singular potentials very few is known. Fournier, Hauray and Mischler \cite{FHM} have tackled the case of the 2-D viscous vortex model, corresponding to $K(x)=\frac{x^\perp}{\parallel x\parallel^2}$ for which no blow-up phenomenon occurs. In the same spirit the sub-critical Keller-Segel model corresponding to $K(x)=\frac{x}{\parallel x\parallel^{2-\varepsilon}}$ for some $\varepsilon >0$ is studied in \cite{GQ}. The methods of both papers are close, and mainly based on some entropic controls. These methods seem to fail for the classical Keller-Segel model we are looking at. However, during the preparation of the manuscript, we received the paper by N. Fournier and B. Jourdain \cite{FJ}, who prove existence and some weak uniqueness by using approximations. Though some intermediate results are the same, we shall here give a very different and much direct approach, at least for existence and some uniqueness. However, we shall use one result in \cite{FJ} to prove a more general uniqueness result.

Also notice that when we replace the attractive potential $K$ by a repulsive one (say $-K$), we find models connected with random matrix theory (like the Dyson Brownian motion).
\bigskip

Our main theorem in this paper is the following

\begin{theorem}\label{thmmain}
Let $M=\{\textrm{there exists at most one pair $i \neq j$ such that $X^i=X^j$}\}$. Then,  
\begin{itemize}
\item for $N \geq 4$ and $\chi < 4 \, \left(1 - \frac{1}{N-1}\right)$, there exists a unique (in distribution) non explosive solution of \eqref{eqsys}, starting from any $x \in M$. Moreover, the process is strong Markov, lives in $M$ and admits a symmetric $\sigma$-finite, invariant measure given by $$\mu(dX^1,...,dX^N) = \Pi_{1\leq i<j \leq N} \, \parallel X^i-X^j\parallel^{-\frac \chi N} \, dX^1 ... dX^N \, ,$$
\item for $N\geq 2$, if $\chi > 4$, the system \eqref{eqsys} does not admit any global solution (i.e. defined on the whole time interval $\mathbb R^+$),
\item for $N\geq 2$, if $\chi = 4$, either the system \eqref{eqsys} explodes  or the $N$ particles are glued in finite time.
\end{itemize}
\end{theorem}

Since later we will be interested in the limit $N \to +\infty$, this theorem is in a sense optimal: for $\chi <4$ we have no asymptotic explosion while for $\chi > 4$ the system explodes. Also notice that in the limiting case $\chi=4$, we have (at least) explosion for the density of the stochastic system and not for the equation \eqref{eqKS}.

The proof of this theorem is partly ``pathwise'', based on comparisons between one dimensional diffusion processes and the behavior of squared Bessel processes, partly based on Dirichlet forms theory and partly based on an uniqueness result for $2$ dimensional skew Bessel processes obtained in \cite{FJ}. The latter is only used to get rid of a non allowed polar set of starting positions which appears when using Dirichlet forms. 

The remaining part of the whole program will be the aim of future works.
\medskip

\section{Study of the system \eqref{eqsys}.} \label{secproofs}

Most of the proofs in this section will use comparison with squared Bessel processes. Let us recall some basic results on these processes.
\begin{definition}\label{defbes}
Let $\delta \in \mathbb R$. The unique strong solution (up to some explosion time $\tau$) of the following one dimensional stochastic integral equation $$Z_t = z + 2 \, \int_0^t \, \sqrt{Z_s} \, dB_s + \delta \, t \, ,$$ is called the (generalized) squared Bessel process of dimension $\delta$ starting from $z\geq 0$.
\end{definition}
In general squared Bessel processes are only defined for $\delta \geq 0$, that is why we used the word generalized in the previous definition. For these processes the following properties are known
\begin{proposition}\label{propbes}
Let $Z$ be a generalized squared Bessel process of dimension $\delta$. Let $\tau_0$ the first hitting time of the origin. 
\begin{itemize}
\item If $\delta <0$, then $\tau_0$ is almost surely finite and equal to the explosion time,
\item if $\delta=0$, then $\tau_0 <+\infty$ and $Z_t=0$ for all $t\geq \tau_0$ almost surely, 
\item if $0<\delta<2$, then $\tau_0 <+\infty$ almost surely and the origin is instantaneously reflecting,
\item if $\delta \geq 2$, then the origin is polar (hence $\tau_0=+\infty$ almost surely).
\end{itemize}
\end{proposition}
For all this see \cite{RY} chapter XI, proposition 1.5. 
\bigskip

Now come back to \eqref{eqsys}. For simplicity we skip the index $N$ in the definition of the process. \\Since all coefficients are locally Lipschitz outside the set $$A=\left\{\textrm{there exists (at least) a pair $i\neq j$ such that $X^i=X^j$}\right\}$$ and bounded when the distance to $A$ is bounded from below, the only problem is the one of collisions between particles. As usual we denote by $\xi$ the lifetime of the process. For simplicity we also assume, for the moment, that the starting point does not belong to $A$, so that the lifetime is almost surely positive.

For $2\leq k \leq N$ we define $K=\{1,...,k\}$ and $\bar K^2=\{(i,j) \in K^2 | i \neq j \}$. We shall say that a $k$-collision occurs at (a random) time $T$ if $X_T^i=X_T^j$ for all $(i,j) \in \bar K^2$, $X_T^l \neq X_T^i$ for all $l > k$. Of course, there is no lack of generality when looking at the first $k$ indices, and we can also assume that at this peculiar time $T$, any other collision involves at most $k$ other particles.\\ In what follows we denote $D^{i,j}=X^i-X^j$, and $$Z^k = \sum_{(i,j)\in \bar K^2} \, \parallel D^{i,j}\parallel^2 \, .$$ Of course a $k$-collision occurs at time $T$ if and only if $T<\xi$ and $Z^k_T=0$.
\medskip

Let us study the process $Z^k$. Applying Ito's formula we get on $t<\xi$
\begin{eqnarray}\label{eqito}
Z^k_t &=& Z^k_0 + 2\sqrt 2 \, \int_0^t \, \sum_{(i,j)\in \bar K^2} \, D^{i,j}_s \, (dB_s^i - dB_s^j) + 4k(k-1) \left(2-\frac \chi N \right)t \\ & & \, - \, \frac{2\chi}{N} \, \int_0^t \, \sum_{(i,j)\in \bar K^2} \, D^{i,j}_s \, \sum \limits_{\underset{l \neq i,j}{l=1}}^N \left(\frac{D^{i,l}_s}{\parallel D^{i,l}_s\parallel^2} + \frac{D^{l,j}_s}{\parallel D^{l,j}_s\parallel^2}\right) \, ds \, . \nonumber
\end{eqnarray}
We denote
\begin{eqnarray*}
dM^k_s &=& \sum_{(i,j)\in \bar K^2} \, D^{i,j}_s \, (dB_s^i - dB_s^j) \\
E^k_s &=& \sum_{(i,j)\in \bar K^2} \, D^{i,j}_s \, \sum \limits_{\underset{l \neq i,j}{l=1}}^N \left(\frac{D^{i,l}_s}{\parallel D^{i,l}_s\parallel^2} + \frac{D^{l,j}_s}{\parallel D^{l,j}_s\parallel^2}\right) \, 
\end{eqnarray*}
the martingale part and the non-constant drift part.
\medskip

\subsection{Investigation of the martingale part $M^k$.} \quad Let us compute the martingale bracket, using the immediate $D^{i,l}=-D^{l,i}$ and $D^{i,l}+D^{l,j}=D^{i,j}$. 

\begin{eqnarray*}
d<M^k>_s &=& \sum \limits_{\underset {(l,m)\in \bar K^2}{(i,j)\in \bar K^2}} \, <D^{i,j}_s \, (dB_s^i - dB_s^j),D^{l,m}_s \, (dB_s^l - dB_s^m)> \\ &=& \sum \limits_{\underset {(l,m)\in \bar K^2}{(i,j)\in \bar K^2}} \, D^{i,j}_s \, D^{l,m}_s \, (\delta_{il} - \delta_{im} - \delta_{jl} + \delta_{jm}) \, ds \\ &=& \sum_{(i,j)\in \bar K^2} \, D^{i,j}_s \, \left(\sum \limits_{\underset{m \neq i}{m \in K}} D^{i,m}_s +\sum \limits_{\underset{l \neq i}{l \in K}} D^{i,l}_s + \sum \limits_{\underset{l \neq j}{l \in K}} D^{l,j}_s + \sum \limits_{\underset{m \neq j}{m \in K}} D^{m,j}_s\right) ds \\ &=& 2 \, \sum_{(i,j)\in \bar K^2} \, D^{i,j}_s \, \left(\sum \limits_{\underset{m \neq i}{m \in K}} D^{i,m}_s + \sum \limits_{\underset{m \neq j}{m \in K}} D^{m,j}_s\right) ds \\ &=& 2k \, \sum_{(i,j)\in \bar K^2} \, \parallel D^{i,j}_s \parallel^2
\end{eqnarray*}
i.e. finally
\begin{equation}\label{martin}
d<M^k>_s = 2k \, Z_s^K \, ds \, .
\end{equation}
According to Doob's representation theorem (applied to $\BBone_{s<\xi} \, d<M^k>_s$), there exists (on an extension of the initial probability space) a one dimensional Brownian motion $W^k$ such that almost surely for $t<\xi$
\begin{equation}\label{eqdoob}
2 \sqrt 2 \, \int_0^t \, \sum_{(i,j)\in \bar K^2} \, D^{i,j}_s \, (dB_s^i - dB_s^j) \, = \, 4\sqrt k \, \int_0^t \, \sqrt{Z^k_s} \, dW_s^k \, .
\end{equation}
\medskip

\subsection{Reduction of the drift term.} \quad In order to study the drift term $E^k_t$ we will divide it into two sums: the first one, $C_t^k$ taking into consideration the $i$ and $j$ in $K$, i.e. the pair of particles which will be directly involved in the eventual $k$-collision; the other one $R_t^k$, involving the remaining indices.

More precisely $E_t^k=C_t^k+R_t^k$ with
\begin{eqnarray*}
C_t^k &=& \sum_{(i,j)\in \bar K^2} \,  \sum \limits_{\underset{l \neq i,j}{l \in K}} \, D^{i,j}_t \, \left(\frac{D^{i,l}_t}{\parallel D^{i,l}_t\parallel^2} + \frac{D^{l,j}_t}{\parallel D^{l,j}_t\parallel^2}\right) \, , \\ R_t^k &=& \sum_{(i,j)\in \bar K^2} \,  \sum_{l=k+1}^N \, D^{i,j}_t \, \left(\frac{D^{i,l}_t}{\parallel D^{i,l}_t\parallel^2} + \frac{D^{l,j}_t}{\parallel D^{l,j}_t\parallel^2}\right) \, .
\end{eqnarray*}
We will deal with $R_t^k$ later. First we ought to simplify the expression of $C_t^k$. Indeed
\begin{eqnarray*}
C_t^k &=& \sum_{(i,j)\in \bar K^2} \,  \sum \limits_{\underset{l \neq i,j}{l \in K}} \, D^{i,j}_t \, \left(\frac{D^{i,l}_t}{\parallel D^{i,l}_t\parallel^2} + \frac{D^{l,j}_t}{\parallel D^{l,j}_t\parallel^2}\right) \, \\ &=& \sum_{(i,l)\in \bar K^2} \, \left(\frac{D^{i,l}_t}{\parallel D^{i,l}_t\parallel^2} \, \sum \limits_{\underset{j \neq i,l}{j \in K}} \, D^{i,j}_t\right) + \sum_{(j,l)\in \bar K^2} \, \left(\frac{D^{l,j}_t}{\parallel D^{l,j}_t\parallel^2} \, \sum \limits_{\underset{i \neq j,l}{i \in K}} D^{i,j}_t\right)\\ &=& \sum \limits_{\underset {i<l}{(i,l)\in \bar K^2}} \, \left(\frac{D^{i,l}_t}{\parallel D^{i,l}_t\parallel^2} \, \sum \limits_{\underset{j \neq i,l}{j \in K}} \, D^{i,j}_t + \frac{D^{l,i}_t}{\parallel D^{l,i}_t\parallel^2} \, \sum \limits_{\underset{j \neq i,l}{j \in K}} \, D^{l,j}_t\right) \, + \\ & & \quad \quad + \, \sum \limits_{\underset {j<l}{(j,l)\in \bar K^2}} \, \left(\frac{D^{l,j}_t}{\parallel D^{l,j}_t\parallel^2} \, \sum \limits_{\underset{i \neq j,l}{i \in K}} \, D^{i,j}_t + \frac{D^{j,l}_t}{\parallel D^{j,l}_t\parallel^2} \, \sum \limits_{\underset{i \neq j,l}{j \in K}} \, D^{i,l}_t\right)
\end{eqnarray*}
so that using again $D^{l,i} D^{l,j} = D^{i,l} D^{j,l}$ the latter is still equal to
\begin{eqnarray*}
&=& \sum \limits_{\underset {i<l}{(i,l)\in \bar K^2}} \, \frac{D^{i,l}_t}{\parallel D^{i,l}_t\parallel^2} \, \left(\sum \limits_{\underset{j \neq i,l}{j \in K}} \, (D^{i,j}_t+D^{j,l}_t)\right) + \sum \limits_{\underset {j<l}{(j,l)\in \bar K^2}} \, \frac{D^{l,j}_t}{\parallel D^{i,l}_t\parallel^2} \, \left(\sum \limits_{\underset{i \neq j,l}{i \in K}} \, (D^{i,j}_t+D^{l,i}_t)\right)
\end{eqnarray*}
and using again $D^{i,l}+D^{l,j}=D^{i,j}$, we finally arrive at
\begin{equation}\label{eqC}
C_t^k \, = \, 2 \, (k-2) \, \times \, \# \{(i,l)\in K^2|i<l\} \, = \, k(k-1)(k-2) \, .
\end{equation}

\subsection{Back to the process $Z^k$.} \quad With the results obtained in \eqref{eqdoob} and \eqref{eqC} we may simplify \eqref{eqito}, writing (still on $t<\xi$)
\begin{equation}\label{eqdist}
Z^k_t = Z^0_t + 4\sqrt k \, \int_0^t \, \sqrt{Z^k_s} \, dW_s^k + 2k(k-1)\left(4 - \frac{k \chi}{N}\right) \, t - \frac {2\chi}{N} \int_0^t R^k_s \, ds \, .
\end{equation}
Hence defining $$V_t^k = \frac{1}{4k} \, Z_t^k$$ the process $V^k$ satisfies
\begin{equation}\label{eqdist2}
dV^k_t = 2 \, \sqrt{V^k_t} \, dW_t^k + (k-1)\left(2 - \frac{k \chi}{2N}\right) \, dt - \frac {\chi}{2kN} R^k_t \, dt \, ,
\end{equation}
i.e. can be viewed as a perturbation of a squared Bessel process of dimension $$\delta = (k-1)\left(2 - \frac{k \chi}{2N}\right) \, ,$$ we shall denote by $U^k$ in the sequel. 
\medskip

\subsection{The case of an $N$-collision.} \quad If $k=N$, $R^N=0$ so that $V^N$ is exactly the squared Bessel process of dimension $\frac{N-1}{2} \, (4-\chi)$. Hence, according to proposition \ref{propbes}
\begin{itemize}
\item if $\chi>4$ there is explosion in finite time for the process $V^N$ (hence for $X$ also),
\item if $\chi =4$, there is an almost sure $N$-collision in finite time, and then all the particles are glued, provided no explosion occurred before for the process $X$, 
\item if $4\left(1 - \frac{1}{N-1}\right) < \chi <4$ there is an almost sure $N$-collision in finite time, provided no explosion occurred before for the process $X$,
\item if $\chi \leq 4\left(1 - \frac{1}{N-1}\right)$ there is almost surely no $N$-collision (before explosion).
\end{itemize} 

In particular we see that the particle system immediately feels the critical value $\chi=4$, in particular explosion occurs in finite time as soon as $\chi>4$. \\ For $4\left(1 - \frac{1}{N-1}\right) < \chi <4$ we know that $V^N$ is instantaneously reflected once it hits the origin, but it does not indicate whether all or only some particles will separate (we only know that they are not all glued). Notice that when $N=2$ this condition reduces to $0<\chi<4$, and then both particles are separated. Hence in this very specific case, there is no explosion (for the distance between both particles) in finite time almost surely, but there are always $2$-collisions.

\subsection{Towards non explosion for $\chi \leq 4\left(1 - \frac{1}{N-1}\right)$.} \quad As we said before, the lifetime $\xi$ is greater than or equal to the first multiple collision time $T$. \\ Since we shall consider $V^k$ as a perturbation of $U^k$, what happens for the latter ?
\begin{itemize}
\item for $\chi > \frac{4N}{k}$, $U^k$ reaches $0$ in finite time a.s. and then explosion occurs,
\item for $\chi = \frac{4N}{k}$, $U^k$ reaches $0$ and is sticked,
\item for $\frac{4N}{k}>\chi > \frac{4N}{k} \left(1-\frac{1}{k-1}\right)$, $U^k$ reaches $0$ and is instantaneously reflected,
\item for $\chi \leq \frac{4N}{k} \, \left(1-\frac{1}{k-1}\right)$, $U^k$ does not hit $0$ in finite time a.s.
\end{itemize}

\begin{lemma}\label{lemk}
For all $3 \leq k \leq N$, it holds $$\frac{4N}{k} \, \left(1-\frac{1}{k-1}\right) \geq 4\left(1 - \frac{1}{N-1}\right) \, .$$
\end{lemma}
\begin{proof}
Introduce the function $$u \mapsto g(u) =\frac{4N}{u} \, \left(1-\frac{1}{u-1}\right)$$ defined for $u>1$. Then $$g'(u) = \frac{4N}{u(u-1)} \, \left(\frac{2-u}{u} + \frac{1}{u-1}\right)$$  is negative on $[2+\sqrt 2 , +\infty[$, so that the lemma is proved for $N \geq k\geq 4$. For $k=3$, it amounts to $\frac{N}{6} \geq \frac{N-2}{N-1}$ which is true for all $N \geq 3$ (with equality for $N=3$ and $N=4$).
\end{proof}

In particular, since $\chi \leq 4\left(1 - \frac{1}{N-1}\right)$, $U^k$ never hits $0$ for $3 \leq k \leq N$, while it reaches $0$ but is instantaneously reflected for $k=2$. What we expect is that the same occurs for $V^k$. \\

In order to prove it, let $T$ be the first multiple collision time. With our convention (changing indices if necessary) there exists some $2\leq k \leq N$ such that $T$ is the first $k$-collision time $T^k$. Note that this does not prevent other $k'$-collisions (with $k' \leq k$) possibly at the same time $T$ for the particles with indices larger than $k+1$. But as we will see this will not change anything. The reasoning will be the same but the conclusion completely different for $k=2$ and for $k\geq 3$.
\smallskip

\subsubsection{No $k$-collisions for $k\geq 3$.} \quad 
Introduce, for $\varepsilon >0$, the random set $$A^k_\varepsilon = \{T=T^k<+\infty \, \textrm{ and } \, \inf_{i \in K \, , \,  l\geq k+1} \, \inf_{t \leq T} \, \parallel D^{i,l}_t\parallel \geq 2 \varepsilon \} \, .$$ It holds $$\{T=T^k<+\infty\} = \bigcup_{\varepsilon \, \in 1/\mathbb N} \, A_\varepsilon^k \, . $$ In particular if $\mathbb P(T=T^k<+\infty)>0$ there exists some $\varepsilon >0$ so that $\mathbb P(A_\varepsilon^k)>0$. \\ We shall see that this is impossible when $k\geq 3$.
\medskip

Indeed recall that $$ R_t^k = \sum_{(i,j)\in \bar K^2} \,  \sum_{l=k+1}^N \, D^{i,j}_t \, \left(\frac{D^{i,l}_t}{\parallel D^{i,l}_t\parallel^2} + \frac{D^{l,j}_t}{\parallel D^{l,j}_t\parallel^2}\right) \, .$$ So on $A_\varepsilon^k$, we have, for $t\leq T$, 
\begin{eqnarray*}
|R_t^k| &\leq& \sum_{(i,j)\in \bar K^2} \,  \sum_{l=k+1}^N \, \parallel D^{i,j}_t\parallel \, \left(\frac{1}{\parallel D^{i,l}_t\parallel^2} + \frac{1}{\parallel D^{l,j}_t\parallel^2}\right)\\ &\leq& \sum_{(i,j)\in \bar K^2} \, \parallel D^{i,j}_t\parallel \, \frac{N-k}{\varepsilon} \\ &\leq& \frac{N-k}{\varepsilon} \, \sqrt{k(k-1)} \, \sqrt{Z_t^k}
\end{eqnarray*}
the latter being a consequence of Cauchy-Schwarz inequality. Thus on $A_\varepsilon^k$, for $t\leq T$
\begin{equation}\label{eqR}
|R_t^k| \leq \frac{2}{\varepsilon} \, (N-k) k \, \sqrt{k-1} \, \sqrt{V_t^k} \, .
\end{equation}
Hence, on $A_\varepsilon^k$ for $t\leq T$ the drift $b^k$ (which is not Markovian) of $V_t^k$ satisfies 
\begin{equation}\label{eqdriftV}
b^k \geq \hat b^k(v)=\frac{(k-1)}{2} \, \left(4 - \frac{N\chi}{k}\right) - \frac{2}{\varepsilon} \, (N-k) k \, \sqrt{k-1} \, \sqrt v \, .
\end{equation}
In particular for any $\theta >0$, $$\hat b^k(v) \geq \frac{(k-1)}{2} \, \left(4 - \frac{N\chi}{k}\right) - \theta$$ provided $v$ is small enough. Thus the hitting time of the origin for the process with drift $\hat b^k$ is larger than the one for the corresponding squared Bessel process (thanks to well known comparison results between one dimensional Ito processes, see e.g. \cite{IW} Chap.6, Thm 1.1), and since this holds for all $\theta$, finally is larger than the one of $U^k$. But as we already saw, $U^k$ never hits the origin for $3\leq k$. Using again the comparison theorem (this time with $b^k$ and $\hat b^k(v)$), $V^k$ does not hit the origin in finite time on $A^k_\varepsilon$ which is in contradiction with $\mathbb P(A^k_\varepsilon)>0$.
\medskip

\subsubsection{About $2$-collisions.} \quad Actually all we have done in the previous sub subsection is unchanged for $k=2$, except the conclusion. Indeed $U^2$ reaches the origin but is instantaneously reflected. So $V^2$ (on $A_\varepsilon^2$) can reach the origin too, but is also instantaneously reflected. Actually using that $$b^k(v) \leq \bar b^k(v)=\frac{(k-1)}{2} \, \left(4 - \frac{N\chi}{k}\right) + \frac{2}{\varepsilon} \, (N-k) k \, \sqrt{k-1} \, \sqrt v \, ,$$ together with the Feller's explosion test, it is easily seen that $V^2$ will reach the origin with a (strictly) positive probability (presumably equal to one, but this is not important for us).

But this instantaneous reflection is not enough for the non explosion of the process $X$, because $X^i$ is $\mathbb R^2$ valued. Before going further in the construction, let us notice another important fact: there are no multiple $2$-collisions at the same time, i.e. starting from $A^c$ the process lives in $M$ at least up to the explosion time $\xi$. Of course this is meaningful provided $N\geq 4$.\\

To prove the previous sentence, first look at $$Y_t=\parallel D_t^{1,2}\parallel^2 + \parallel D_t^{3,4}\parallel^2 \, ,$$ assuming that $Y_T=0$ and that no other $2$-collision happens at time $T$. It is easily seen that (just take care that we had an extra factor $2$ in our definition of $Z_t^k$)
\begin{eqnarray*}
Y_t &=& Y_0 + 2\sqrt 2 \, \int_0^t  \, \left(D^{1,2}_s \, (dB_s^1 - dB_s^2)+D^{3,4}_s \, (dB_s^3 - dB_s^4)\right) + 8 \left(2-\frac \chi N \right)t \\ & & \, - \, \frac{\chi}{N} \, \int_0^t \, \left(D^{1,2}_s \, \sum \limits_{\underset{l \neq 1,2}{l=1}}^N \left(\frac{D^{1,l}_s}{\parallel D^{1,l}_s\parallel^2} + \frac{D^{l,2}_s}{\parallel D^{l,2}_s\parallel^2}\right) \, + \, D^{3,4}_s \, \sum \limits_{\underset{l \neq 3,4}{l=1}}^N \left(\frac{D^{3,l}_s}{\parallel D^{3,l}_s\parallel^2} + \frac{D^{l,4}_s}{\parallel D^{l,4}_s\parallel^2}\right)\right) \, ds \, ,
\end{eqnarray*}
so that, defining $V_t=Y_t/4$ we get that $$dV_t = 2 \, \sqrt{V_t} \,  dW_t + 2\left(2-\frac \chi N \right)dt + \, R_t \, dt$$ where $R_t$ is a remaining term we can manage just as we did for $Z_t^k$. Since for $N\geq 4$, $2\left(2-\frac \chi N \right)\geq 2$, $V_t$, hence $Y_t$ does not hit the origin. Notice that if we consider $k(\geq 2)$ $2$-collisions, the same reasoning is still true, just replacing $4$ by $2k$, the final equation being unchanged except for $R_t$. 
\medskip

\subsubsection{Non explosion.} \quad According to all what precedes what we need to prove is the existence of the solution of \eqref{eqsys} with an initial configuration satisfying $X_0=x$ with $x^1=x^2$, all other coordinates being different and different from $x^1=x^2$. Indeed, on $\xi < +\infty$, $X_\xi \in \delta M$ the set of particles with exactly two glued particles, so that if we can prove that starting from any point of $\delta M$, we can build a strong solution on an interval $[0,S]$ for some strictly positive stopping time $S$, it will show that $\xi=+\infty$ almost surely. However we will not be able to prove the existence of such a strong solution. Actually we think that it does not exist. We will thus build some weak solution and show uniqueness in some specific sense.

This will be the goal of the next sections.
\bigskip

\section{Building a solution.}\label{secsol}

\subsection{Existence of a weak solution.}\label{secweak}

Writing $$M = \cup_{i<j} \, \cap_{k\neq l \, ; \, l\neq i,j} \, \{X^k\neq X^l\}$$ we see that $M$ is an open subset of $\mathbb R^{2N}$. 

Recall that $x \in \delta M$ means that exactly two coordinates coincide (say $x^1=x^2$), all other coordinates being distinct and distinct from $x^1$. We may thus define $$d_x=\min\{i \geq 3 \, ; i \neq j \, ; \, \, j=1,...,N \, ; \, \parallel x^i -x^j\parallel\} \, > \, 0 \, ,$$ so that 
\begin{equation}\label{eqomega}
\Omega_x =\Pi_{j=1}^N \, B(x^j,d_x/2) \, \subset \, M \, ,
\end{equation}
and points $y \in \Omega_x\cap \delta M$ will satisfy $y^1=y^2$. If $x \notin \delta M$, we may similarly define $d_x=\min\{i \neq j \, ; \, j=1,...,N \, ; \, \parallel x^i -x^j\parallel\}$ and then $\Omega_x$. In all cases the balls $B(.,.)$ are the open balls. Now if $K$ is some compact subset of $M$ we can cover $K$ by a finite number of sets $\Omega_x$, so that for any measure $\mu$, a function $g$ belongs to $\mathbb L^1_{loc}(M,\mu)$ if and only if $g \in \mathbb L^1(\Omega_x,\mu)$ for all $x$ in $M$.
\medskip

The natural measure to be considered is
\begin{equation}\label{eqmu}
\mu(dX^1,...,dX^N) = \Pi_{1\leq i<j \leq N} \, \parallel X^i-X^j\parallel^{-\frac \chi N} \, dX^1 ... dX^N \, ,
\end{equation}
since it is, at least formally, the symmetric measure for the system \eqref{eqsys}.

It is clear that for $x \notin \delta M$, $\mu$ is a bounded measure on $\Omega_x$. When $x \in \delta M$, say that $x^1=x^2$ and perform the change of variables $$Y^1=X^1-X^2 \quad , \quad Y^2=X^1+X^2 \, .$$ In restriction to $\Omega_x$, $\mu$ can be written $$\mu(dX^1,...,dX^N) = C(N,x) \, \parallel Y^1\parallel^{-\frac \chi N} \, dY^1 \, dY^2 \, dX^3 ...dX^N \, ,$$ hence is a bounded measure on $\Omega_x$ provided $\chi < 2N$ just looking at polar coordinates for $Y^1$. In this case it immediately follows that $\mu$ is a $\sigma$ finite measure on $M$. Also remark that if $f$ is compactly supported by $K$ and belongs to $\mathbb L^2(d\mu)$ then it belongs to $\L^2(dX)$ and $$\int_K \, f^2 \, dX \leq \sup_K (\parallel Y^1\parallel^{\frac \chi N})  \, \int_K \, f^2 \, d\mu \, .$$  
But we can say much more. 
\medskip

To this end consider the symmetric form 
\begin{equation}\label{eqform}
\mathcal E(f,g) = \int_M \, <\nabla f , \nabla g> \, d\mu \quad , \quad f,g \in C^\infty_0(M) \, .
\end{equation}
First we check that this form is closable in the sense of \cite{Fuk}. To this end it is enough to show that it is a closable form when restricted to functions $f,g \in C^\infty_0(\Omega_x)$ for all $x \in M$. If $x \notin \delta M$ the form is equivalent to the usual scalar product on square Lebesgue integrable functions, so that it is enough to look at $x \in \delta M$.

Hence let $f_n$ be a sequence of functions in $C^\infty_0(\Omega_x)$, converging to $0$ in $\L^2(d\mu)$ and such that $\nabla f_n$ converges to some vector valued function $g$ in $\L^2(d\mu)$. What we need to prove is that $g$ is equal to $0$. To this end consider a vector valued function $h$ which is smooth and compactly supported in $\Omega_x \cup \{d(.,\delta M)>\varepsilon\}$ for some $\varepsilon >0$. Then a simple integration by parts shows that $$\int \, <g,h> \, d\mu = \lim_n \, \int \, <\nabla f_n,h> \, d\mu = \lim_n \, \int \, f_n \, H \, d\mu$$ for some $H \in \L^2(d\mu)$ so is equal to $0$. Hence $g$ vanishes almost surely on $\Omega_x \cup \{d(.,\delta M)>\varepsilon\}$, for all $\varepsilon >0$, so that $g$ is $\mu$-almost everywhere equal to $0$.

By construction, $\mathcal E$ is regular and local. Hence, its smallest closed extension $(\mathcal E, \mathcal D(\mathcal E))$ is a Dirichlet form, which is in addition regular and local. According to Theorem 6.2.2. in \cite{Fuk}, there exists a $\mu$-symmetric diffusion process $X_.$ whose form is given by $\mathcal E$. Notice that, integrating by parts, we see that the generator of this diffusion process coincides with the generator $L$ given by  

\begin{equation}\label{eqgenerator}
L=\sum_{i=1}^N \, \Delta_{x^i} - \, \frac \chi N \, \sum_{i=1}^N \, \left(\sum_{i\neq j} \, \frac{x^i-x^j}{\parallel x^i - x^j\parallel^2}\right) \, \nabla_{x^i}
\end{equation}
for the functions $f$ in $C^\infty_0(M)$ such that $Lf \in \mathbb L^2(\mu)$. This is a core for the domain $D(L)$. The Dirichlet form theory tells us that once $f \in D(L)$, $f(X_t)-f(X_0) -\int_0^t \, Lf(X_s) ds$ is a $\mathbb P_x$ martingale for quasi every starting point $x$, i.e. for all $x$ out of some subset $E$ of $M$ which is of zero $\mu$-capacity. 

But remark that for any $x \in M-E$ and $t>0$, the transition kernel $p_t(x,.)$ of the Markov semi-group is absolutely continuous with respect to $\mu$. Indeed using the local Malliavin calculus as in \cite{Cat1} (or elliptic standard results), this transition kernel has a (smooth) density w.r.t. Lebesgue measure (hence w.r.t. $\mu$) on each open subset of $M \cap \{d(.,\delta M)>\varepsilon\}$ for any $\varepsilon >0$. Hence if $\mu(A)=0$, $\mu(A \cap \{d(.,\delta M)>\varepsilon\})=0$ for all $\varepsilon >0$ so that $p_t(x,A \cap \{d(.,\delta M)>\varepsilon\})=0$ and finally using monotone convergence, $p_t(x,A)=0$.\\ Since $p_t(x,.)$ is absolutely continuous w.r.t. $\mu$, we deduce from Theorem 4.3.4 in \cite{Fuk} that the sets of zero $\mu$ capacity are exactly the polar sets for the process.
\smallskip

Note that the function $x \mapsto x$ does not belong to $D(L)$, so that we cannot use the previous result. Nevertheless 
\begin{lemma}\label{lemcoord}
Assume that $\chi<N$. Then for all $x \in M-E$ and all $i=1,...,N$, $$X^i_t-X_0^i-\int_0^t \, \frac \chi N \,  \left(\sum_{i\neq j} \, \frac{X_s^i-X_s^j}{\parallel X_s^i - X_s^j\parallel^2}\right) ds$$ is a $\mathbb P_x$ martingale, and actually is $\mathbb P_x$ almost surely equal to $\sqrt 2 \, B^i_t$ for some Brownian motion.
\end{lemma}
\begin{proof}
To prove the lemma, for all $x \in M-E$ it is enough to show the martingale property starting from $x$ up to the exit time $S(x)$ of $\Omega_x$ (since $X_{S(x)} \in M-E$ because $E$ is polar, see the discussion above). In the sequel, for notational convenience, we do not write the exit time $S(x)$ (all times $t$ have to be understood as $t\wedge S(x)$) and we simply write $M$ instead of $\Omega_x$.

To show this result it is actually enough to look locally in the neighborhood of a point $x \in \delta M$ such that $x^1=x^2$, and with our previous notation to look at both coordinates of $y=x^1-x^2$. 
Indeed $x \mapsto x^1+x^2$ belongs (at least locally) to $D(L)$ as well as all other coordinates $x^j$ for $j\geq 3$. 

Let $g^j(x)=y_j$ for $j=1,2$ be the coordinate application of $y=x^1-x^2$. Clearly $Lg^j \in \mathbb L^p(\mu)$ for $p<2-\frac \chi N$, hence belongs to $\mathbb L^1$ thanks to our assumption on $\chi/N$. Introduce the function defined on $\mathbb R$ by, $$h_\varepsilon(u)=\sin^2\left(\frac{\pi \, u}{2 \varepsilon}\right) \mathbf 1_{|u|\leq \varepsilon} + \, \mathbf 1_{|u| > \varepsilon} \, .$$ $h$ is of $C^2$ class except at $|u|=\varepsilon$. Now define $v_\varepsilon(x)=g^1(x) \, h_\varepsilon(g^1(x))$. We have $$L(v_\varepsilon)(x)=\left[4h'_\varepsilon(g^1)+ 2 \, g^1 \, h''_\varepsilon(g^1) - \frac \chi N \, \left(2 h_\varepsilon(g^1) \, \frac{g^1}{|g^1+g^2|^2} + 2h'_\varepsilon(g^1) + \, R_\varepsilon\right)\right](x)$$ the remaining term $R_\varepsilon$ corresponding to the interactions with particles $x^j$ for $j\geq 3$.

For all $\varepsilon >0$, $v_\varepsilon$ thus belongs to $D(L)$ and $v_\varepsilon(X_t)-v_\varepsilon(X_0) -\int_0^t \, Lv_\varepsilon(X_s) ds$ is a $\mathbb P_x$ martingale, with brackets $4 \, \int_0^t \, |\nabla v_\varepsilon (X_s)|^2 \, ds$ for all $x \in M-E$.

But it is easily seen that $Lv_\varepsilon$ converges to $Lg^1$ in $\mathbb L^1(\mu)$ as $\varepsilon \to 0$. Since $v_\varepsilon$ converges to $g^1$ in $\mathbb L^1(\mu)$ too, we deduce that 
\begin{equation}\label{eqmart}
\E_\mu\left(g^1(X_{t+h})-g^1(X_t) - \int_t^{t+h} Lg^1(X_s) ds | \mathcal F_t\right)=0 \quad \textrm {for all $t\geq 0$ and $h\geq 0$,}
\end{equation}
 where $\mathcal F_t$ denotes the natural filtration on the probability space, since the same property is true for $v_\varepsilon$. Similarly the brackets converge to $4t$. Since the same holds for $g^2$, we get the desired result  $\mathbb P_\mu$ almost surely. Actually this result holds true $\mathbb P_x$ almost surely for $\mu$ almost all $x\in M$.

But since $p_t(x,.)$ is absolutely continuous w.r.t. $\mu$ for $t>0$, it immediately follows using the Markov property at time $t$, that \eqref{eqmart} is true $\mathbb P_x$ a.s. for all $x\in M-E$, but only for $t>0$. Hence for all $x\in M-E$ and all $t>0$, $N_t^s =g^1(X_{t+s})-g^1(X_t) - \int_t^{t+s} Lg^1(X_u) du$ is a  martingale defined on $[t,+\infty[$, whose bracket is given by $4(s-t)$, i.e. is ($2$ times) a Brownian motion. In particular for a fixed $t$, $(N_t^s)_{0<s\leq t}$ is bounded in $\mathbb L^2(\mathbb P_x)$. Up to a sub-sequence it is thus weakly convergent in $\L^2(\mathbb P_x)$ as $s \to 0$ so that $N_t^0=N_t$ is well defined $\mathbb P_x$ a.s., and satisfies \eqref{eqmart} for all $t\geq 0$ this time. Thus it is a martingale with a linear bracket, i.e. $2$ times a Brownian motion.
\end{proof}
\medskip

The previous lemma shows that the diffusion $X_.$ is simply the $\mu$ symmetric solution  of \eqref{eqsys} killed when it hits the boundary $\partial M$.
\medskip

Assume in addition that $\chi \leq 4 \left(1 - \frac{1}{N-1}\right)$. Then the previous diffusion process never hits $\partial M$ since the latter is exactly the set where either some $k$-collision occurs for some $k\geq 3$ or at least two $2$-collisions occur at the same time. So it is actually the unique $\mu$-symmetric Markov diffusion defined on $\bar M$ solving \eqref{eqsys}. Indeed we could associate to any markovian  extension of $(\mathcal E,C_0^\infty(M))$ another diffusion process, which would coincide with the previous one up to the hitting time of $\partial M$ which is almost surely infinite. We have thus obtained

\begin{theorem}\label{thmweak}
Assume that $\chi \leq 4 \left(1 - \frac{1}{N-1}\right)$ and that $N \geq 4$. \\ Then there exists a unique $\mu$-symmetric (see \eqref{eqmu}) diffusion process $(X_t,\mathbb P_x)$ (i.e. a Hunt process with continuous paths), defined for $t\geq 0$ and $x \in M-E$ where $E \subset M$ is polar (or equivalently of $\mu$ capacity equal to $0$) such that for all $f \in C_0^\infty(\mathbb R^{2N})$, $$f(X_t)-f(x) - \int_0^t Lf(X_s) ds$$ is a $\mathbb P_x$ martingale (for the natural filtration) with $L$ given by \eqref{eqgenerator}. Furthermore $X_.$ lives in $M$ (never hits $\partial M$).
\end{theorem}
\begin{proof}  
As for the previous lemma, it is enough to work locally in the neighborhood of the points in $\delta M$ and to look at the new particles $(y=x^1-x^2,z=x^1+x^2,x^3,...,x^N)$. Let $f \in C_0^\infty(M)$ be written in these new coordinates. Using a Taylor expansion in $y$ ($z$ and all the others $x^j$ being fixed) and the fact that if the partial derivatives at $y=0$ of a smooth function of $y$ are vanishing, then this function belongs to the domain of the generator, we see that proving the martingale property for $f$ amounts to the corresponding martingale property for smooth functions $g$ written as $g(y,z,x^j)=y \, h(z,x^j)$ i.e. amounts to the previous lemma (and of course the remaining particles for which there is no problem). 

It remains to extend the martingale property we proved to hold for $f \in C_0^\infty(M)$ to  $f \in C_0^\infty(\mathbb R^{2N})$. Take $f \in C_0^\infty(\mathbb R^{2N})$ and define $S_\varepsilon$ as the first time the distance $d(X_.,\partial M)$ is less than $\varepsilon$. Then replacing $f$ by some $f_\varepsilon \in C_0^\infty(M)$ which coincides with $f$ on $d(y,\partial M)\geq \varepsilon$, we see that  $f(X_{t\wedge S_\varepsilon})-f(x) - \int_0^{t\wedge S_\varepsilon} Lf(X_s) ds$ is a $\mathbb P_x$ martingale. Since $S_\varepsilon$ growths to infinity the conclusion follows from Lebesgue theorem.
\end{proof}
\bigskip

\begin{remark} \label{remweak}

The main disadvantage of the previous construction is that it is not explicit and that it does not furnish a solution starting from all $x\in M$ but only for all $x$ except those in some unknown polar set. In particular, proving the regularity of the Markov transition kernels up to $\delta M$ requires additional work. The advantage is that if we require $\mu$-symmetry, we get uniqueness of the diffusion process.

This Theorem is to be compared with Theorem 7 in \cite{FJ}, where existence of a weak solution is shown by using approximation and tightness, in the same $\chi \leq 4 \left(1 - \frac{1}{N-1}\right)$ case (take care of the normalization of $\chi$ which is not the same here and therein). Note that the result in \cite{FJ} is concerned with existence starting from some initial absolutely continuous density and does not furnish a diffusion process.
\hfill $\diamondsuit$
\end{remark}
\medskip

\subsection{Existence and uniqueness of a weak solution.}\label{secstrong} \quad In this subsection we assume that $\chi \leq 4 \left(1 - \frac{1}{N-1}\right)$ and that $N \geq 4$. Our aim is to build a solution starting from any point in $M$, i.e. to get rid of the polar set $E$ in the previous sub-section. The construction will be very similar (still using Dirichlet forms) but we shall here use one result in \cite{FJ}, namely the uniqueness result for a $2$ dimensional Bessel process.

We start with an important lemma
\begin{lemma}\label{lemlocal}
Let $\mathbb P_x$ be the solution of \eqref{eqsys} built in Theorem \ref{thmweak} and starting from some allowed point $x$. Then $$\int_0^{+\infty} \, \BBone_{\delta M}(X_s) \, ds \, = \, 0 \, , \quad \mathbb P_x \, \, a.s.$$
\end{lemma}
\begin{proof}
We can cover $\delta M$ by an enumerable union of $\Omega_y$ ($y \in \delta M$). It is thus immediate that the lemma will be proved once we prove that $$\int_0^{+\infty} \, \BBone_{\delta M\cap \Omega_y}(X_s) \, ds \, = \, 0 \, , \quad \mathbb P_x \, \, a.s.$$
But we have seen in the previous section that, when the process is in some $\Omega_y$ (where say $y^1=y^2$), the process $\parallel D_t^{1,2}\parallel^2$ is larger than or equal to the square of a Bessel process $U_t$ of index $\delta$ strictly between $0$ and $2$. But (see \cite{RY} proof of proposition 1.5 p.442), the time spent at the origin by the latter is equal to $0$, i.e. $\int_0^{+\infty} \, \BBone_{U_s=0} \, ds = 0$ almost surely. The same necessarily holds for $D^{1,2}$, hence the result. 
\end{proof}
\medskip

We intend now to prove some uniqueness, when starting from a point in $\delta M$. Actually, using some standard tools of concatenation of paths, it is enough to look at the behavior of our process starting at some $y \in \delta M$  with $y^1=y^2$, up to the exit time of $\Omega_y$ (or some open non empty subset of $\Omega_y$). In this case the only difficulty is to control the pair $(X_.^1,X_.^2)$ since all other coordinates are defined through smooth coefficients. Of course writing $$D_t^{1,2}=X_t^1 - X_t^2 \quad , \quad S_t^{1,2}=X_t^1+X_t^2$$ we have that
\begin{equation}\label{eqD}
dD_t^{1,2}= 2 \, dW_t^1 - \frac{2\chi}{N} \, \frac{D_t^{1,2}}{\parallel D_t^{1,2}\parallel^2} \, dt + b_1(X_t) \, dt
\end{equation}
and
\begin{equation}\label{eqS}
dS_t^{1,2}= 2 \, dW_t^2  + b_2(X_t) \, dt
\end{equation}
where $b_1$ and $b_2$ are smooth functions (in $\Omega_y$), $W^1$ and $W^2$ being two independent $2$ dimensional Brownian motions.

Define $\bar \Omega_y$ as we defined $\Omega_y$ (see \eqref{eqomega}, but replacing $d_y/2$ by $d_y/4$ and consider a smoothed version $\eta$ of the indicator of $\bar \Omega_y$ i.e. a smooth non negative function such that $$\BBone_{\bar \Omega_y} \leq \eta \leq \BBone_{\Omega_y} \, .$$ We may extend all coefficients (except $D/\parallel D\parallel^2$) as smooth compactly supported functions outside $\Omega_y$, and replace $D/\parallel D\parallel^2$ by $\eta(X)\,  D/\parallel D\parallel^2$. If we can show uniqueness for this new system we will have shown uniqueness up to the exit time of $\bar \Omega_y$ for \eqref{eqD}, \eqref{eqS} and the remaining part of the initial system.
\medskip

Hence our problem amounts to the following one: prove uniqueness for $Y=(D,S,\bar X) \in \R^2\times \R^2\times \R^{2(N-2)}$ solution of
\begin{eqnarray}\label{eqnewsys}
dD_t&=& 2 \, dW_t - \frac{2\chi}{N} \, \frac{D_t}{\parallel D_t\parallel^2} \, dt + b(D_t,S_t,\bar X_t) \, dt \, , \nonumber\\ dS_t&=& 2 \, dW'_t + b'(D_t,S_t,\bar X_t) \, dt \, ,\\ d\bar X_t&=& \sqrt 2 \, d\bar B_t + \bar b(D_t,S_t,\bar X_t) \, dt \, , \nonumber
\end{eqnarray}
where $b,b',\bar b$ are smooth and compactly supported in $\mathbb R^{2N}$. 
\medskip

Thus, after a standard Girsanov transform, we are reduced to prove uniqueness for
\begin{eqnarray}\label{eqnewsys2}
dD_t&=& 2 \, dW_t - \frac{2\chi}{N} \, \frac{D_t}{\parallel D_t\parallel^2} \, dt  , \nonumber\\ dS_t&=& 2 \, dW'_t \, ,\\ d\bar X_t&=& \sqrt 2 \, d\bar B_t \, , \nonumber
\end{eqnarray}
hence for $D_.$. $U_.= D_./2$ is some type of $2$-dimensional skew Bessel process with dimension $\chi/2N$ (see \cite{Blei} for the one dimensional version). Its squared norm $|U_.|^2$ is a squared Bessel process of  dimension $\delta = 2 - \frac{\chi}{N}$, so that the origin is not polar for the process $D_.$. 
\smallskip

As we did in the previous sub-section, we can directly prove the existence and uniqueness of a symmetric Hunt process (here the reference measure is $|D|^{-\chi/N} \, dD$) using the associated Dirichlet form, and since the origin is not polar, we know the existence of a solution starting from $D_0=0$. Here we only need $\chi < N$, but for the whole construction our initial assumption on $\chi$ is required. Finally we can check that the occupation time formula of Lemma \ref{lemlocal} is still true.
\medskip

But as before, if now we have existence starting from every initial point, we only have uniqueness in the sense of symmetric Markov processes. To get weak uniqueness we can use polar coordinates: the squared norm is a squared Bessel process, so that strong uniqueness holds (with the corresponding dimension we are looking at); the polar angle is much tricky to handle. This is the main goal of Lemma 19 in \cite{FJ}, and the final weak uniqueness then follows from the proof of Theorem 17 in \cite{FJ} and the occupation time formula.
\medskip

\begin{remark}\label{remfin}
It can be noticed than this result is out of reach of the method developed by Krylov and R\"{o}ckner in \cite{KR} for a general Brownian motion plus drift $b$, since it requires that $b \in \L^p(dX)$ for some $p>2$. Also notice that one cannot use standard Girsanov transform for solving \eqref{eqnewsys2}, since for a $2$-dimensional Brownian motion starting from the origin, $$\int_0^t \, \frac{1}{|B_s|^2} \, ds = +\infty  \quad a.s. \quad \textrm{for all $t>0$} \, ,$$ see \cite{RY}. \hfill $\diamondsuit$
\end{remark}

\end{document}